\documentclass[preprint,12pt]{elsarticle}

\usepackage{amssymb}
\usepackage{graphicx}
\usepackage{enumerate}
\usepackage{multirow}
\usepackage{amsmath,color}
\usepackage{hyperref}
\usepackage{comment}
 
\DeclareMathOperator{\fuse}{F}

\DeclareMathOperator{\AGraph}{\Gamma_{\mathcal{A}}}
\DeclareMathOperator{\EGraph}{\Gamma_{\mathcal{E}}}
\DeclareMathOperator{\correspond}{\sideset{_A}{_E}{\mathop{\longleftrightarrow}}\ }

\newtheorem{thm}{Theorem}[section]
\newtheorem{cor}[thm]{Corollary}
\newtheorem{lemma}[thm]{Lemma}
\newtheorem{prop}[thm]{Proposition}

\numberwithin{equation}{section}

\newtheorem{exam}{Example}
\newproof{proof}{Proof}
\newcommand{\BM}{Bannai-Muzychuk criterion}

\newcommand{\I}{\mathcal{I}}
\newcommand{\J}{\mathcal{J}}

\def\Col{\mathop{\rm Col }\nolimits}

\begin{document}

\begin{frontmatter}

\title{Characterizations of amorphic schemes and fusions of pairs}

\author[1]{Edwin R. van Dam}
\ead{Edwin.vanDam@uvt.nl}

\author[2,3]{Jack H. Koolen}
\ead{koolen@ustc.edu.cn}

\author[4]{Yanzhen Xiong\corref{cor1}}
\ead{xyz920805@ustc.edu.cn}

\cortext[cor1]{Corresponding author}

\affiliation[1]{organization={Department of Econometrics and O.R., Tilburg University},
city={Tilburg},
country={The Netherlands}}

\affiliation[2]{organization={School of Mathematical Sciences, University of Science and Technology of China},
city={Hefei, Anhui},
country={PR China}}

\affiliation[3]{organization={Wen-Tsun Wu Key Laboratory of CAS},
city={Hefei,Anhui},
country={PR China}}

\affiliation[4]{organization={Department of Mathematics, National University of Defense Technology},
city={Changsha, Hunan},
country={PR China}}
            
\begin{abstract}
An association scheme is called amorphic if every possible fusion of relations gives rise to a fusion scheme. We call a pair of relations fusing if fusing that pair gives rise to a fusion scheme. We define the fusing-relations graph on the set of relations, where a pair forms an edge if it fuses. We show that if the fusing-relations graph is connected but not a path, then the association scheme is amorphic. As a side result, we show that if an association scheme has at most one relation that is neither strongly regular of Latin square type nor strongly regular of negative Latin square type, then it is amorphic.
\end{abstract}

\begin{keyword}
association schemes, strongly regular graphs, (negative) Latin square type, amorphic, fusion schemes

\textit{AMS subject classification:} {05E30}
\end{keyword}

\end{frontmatter}

\section{Introduction}

Amorphic association schemes are schemes where every possible fusion of relations gives rise to a fusion scheme. It follows easily that every possible union of relations in such a scheme is strongly regular, and hence no relation, or union of relations generates the whole association scheme.

Amorphic association schemes were first studied by Gol'fand, Ivanov, and Klin \cite{GIK}. Ivanov showed, among other results, that all relations in an amorphic $d$-class association scheme with $d \ge 3$ are strongly regular of Latin square type or negative Latin square type.

Ito, Munemasa, and Yamada \cite{IMY} considered the eigenmatrices of amorphic schemes, showed that they are self-dual, and that any scheme in which the relations are either all strongly regular of Latin square type, or strongly regular of negative Latin square type is amorphic. 

Ivanov in fact conjectured that any scheme in which all relations are strongly regular must be amorphic. This conjecture was disproven in \cite{D2,D3} and resulted in many interesting constructions of --- amorphic and non-amorphic --- association schemes in which all relations are strongly regular \cite{DX,F,F2,IM,IM2}.
For an extensive overview of other results on amorphic schemes, we refer to \cite{vDM2010}.

In this paper, we aim to find conditions for association schemes to be amorphic by considering only fusions of pairs of relations, and dually, fusions of pairs of idempotents.  

We say that a tuple of non-trivial relations in an association scheme {\em fuses} if fusing that tuple of relations results in an association scheme. Dually, we say that a tuple of idempotents  {\em fuses} if fusing that tuple results in an association scheme.

Let $\mathcal{R}$ be a $d$-class association scheme with relations $A_0, A_1, \ldots, A_d$ and idempotents $E_0, E_1, \ldots, E_d$. 
The {\em fusing-relations graph} $\AGraph(\mathcal{R})$ of $\mathcal{R}$ is the graph with vertex set $\{1,2,\ldots,d\}$ such that $i$ and $j$ are joined by an edge if the pair $\{A_i,A_j\}$ fuses. Dually, in the {\em fusing-idempotents graph} $\EGraph(\mathcal{R})$, $i$ and $j$ are adjacent if $\{E_i,E_j\}$ fuses.
It is clear that if $\mathcal{R}$ is amorphic, then $\AGraph(\mathcal{R})$ and $\EGraph(\mathcal{R})$ are both complete graphs.

\begin{exam}\label{ex:clique+vertex}
    The non-amorphic $4$-class primitive counterexamples of  Ivanov's conjecture by Van Dam \cite{D3} (and other such counterexamples by Ikuta and Munemasa \cite{IM}) have eigenmatrix of the form
$$P ={\small \begin{bmatrix} 
1 & k_1 & k_2 & k_2 & k_2 \\
1 & b_1 & a_2 & a_2 & a_2 \\
1 & a_1 & a_2 & b_2 & b_2 \\
1 & a_1 & b_2 & a_2 & b_2 \\
1 & a_1 & b_2 & b_2 & a_2 \\
\end{bmatrix}}.$$
Both fusing-relations graph and fusing-idempotents graph are $K_3 \sqcup K_1$, a triangle plus an isolated vertex. The latter represents the first relation/idempotent. 
\end{exam}

Also the next --- imprimitive --- example has a clique plus an isolated vertex as fusing-relations graph.
We will show later that this example is extremal in the number of edges.

\begin{exam}\label{ex:general-clique+vertex}
    Take $m$ copies of a $(d-1)$-class amorphic association scheme on $v$ vertices and add a complete multipartite relation between the copies, i.e., $\mathcal{R}$ is the wreath product of a $1$-class association scheme on $m$ vertices and an amorphic scheme.
    Then the eigenmatrix of $\mathcal{R}$ is of the form
    \[  P =
    {\small \begin{bmatrix}
        1 & k_1 & k_2 & \cdots & k_{d-1} & (m-1)v \\
        1 & b_1 & a_2 & \cdots & a_{d-1} & 0 \\
        1 & a_1 & b_2 & \cdots & a_{d-1} & 0 \\
        1 & \vdots & \vdots & \ddots & \vdots & \vdots \\
        1 & a_1 & a_2 & \cdots & b_{d-1} & 0 \\
        1 & k_1 & k_2 & \cdots & k_{d-1} & -v 
    \end{bmatrix}}. 
    \] 
    When $d\geq 4$, the fusing-idempotents graph and fusing-relations graph of $\mathcal{R}$ are both $K_{d-1} \sqcup K_1$. 
    Note that the added relation cannot fuse with any other, because then also the ``new'' idempotent fuses with some other, which would imply that $k_i-v$ equals $a_i$ or $b_i$, but this is never the case for Latin square type graphs or negative Latin square type graphs.
\end{exam}

Our final example shows that connectedness of the fusing-relations graph is not sufficient to guarantee that the scheme is amorphic.

\begin{exam}\label{ex:path}
Let $\mathcal{R}$ be the repeated wreath product of $1$-class association schemes.
For example, the eigenmatrix for the wreath product $K_{n_1} \wr K_{n_2} \wr K_{n_3} \wr K_{n_4} \wr K_{n_5}$ is
$$P={\tiny \begin{bmatrix}
1 & n_5n_4n_3n_2(n_1-1) & n_5n_4n_3(n_2-1) & n_5n_4(n_3-1) & n_5(n_4-1) & n_5-1 \\
1 & -n_5n_4n_3n_2 & n_5n_4n_3(n_2-1) & n_5n_4(n_3-1) & n_5(n_4-1) & n_5-1 \\
1 & 0& -n_5n_4n_3 &  n_5n_4(n_3-1) & n_5(n_4-1) & n_5-1 \\
1 & 0& 0& -n_5n_4 &   n_5(n_4-1) & n_5-1 \\
1 & 0& 0& 0& -n_5 &   n_5-1 \\
1 & 0& 0& 0& 0 &   -1 \\
\end{bmatrix}}.$$
Both fusing-relations graph and fusing-idempotents graph are $P_5$, i.e., a path of length $4$, and more generally $P_d$. 
\end{exam}

The above examples turn out to be extremal. Indeed, our main result, Theorem \ref{thm:mainresult}, is that if the fusing-relations graph $\AGraph(\mathcal{R})$ is connected, but not a path, then $\mathcal{R}$ is amorphic.

Our paper is further organized as follows. 
In Section \ref{sec:pre}, we give definitions, recall a result about common eigenspaces of strongly regular graphs, and prove a very practical Lemma~\ref{lem:rowscolsP} about the eigenmatrices of association schemes. In Section~\ref{sec:pairs}, which is our basic starting point, we show that if all pairs of relations fuse, then the scheme is amorphic (Theorem \ref{thm:pairs}). 
In Section \ref{sec:latinsquaresrelations}, we focus on strongly regular graphs of (negative) Latin square type. Here we prove that if all but at most one relation is of Latin square type or negative Latin square type, then the scheme is amorphic. This Theorem \ref{thm:LStype} strengthens the earlier mentioned result by Ito, Munemasa, and Yamada \cite{IMY}, and is relevant both in itself and in the further sections.
In Section \ref{sec:structurefpg}, we prove a lemma about the fusing-relations graph of the fusion scheme. Then, in Section \ref{sec:notapath}, we obtain our main result in Theorem \ref{thm:mainresult}.
Next, we dualize all our tools and results in Section \ref{sec:dualization}, and finish with final remarks and open problems in Section \ref{sec:finalremarks}.

\section{Preliminaries}\label{sec:pre}

\subsection{Association schemes}

Let $X$ be a finite set. For each $i=0,1,\ldots,d$, let $A_i$ be a square $01$-matrix whose rows and columns are indexed by $X$. The configuration $\mathcal{R} = \{A_i : i= 0,1,\ldots,d\}$ on $X$ is called a --- symmetric --- {\em $d$-class association scheme} if there are {\em intersection numbers} $p_{ij}^{h}$ such that

\begin{itemize}
    \item $A_0 = I$,
    \item $A_0+A_1+\cdots+A_d=J$, 
    \item $A_i^{\top}=A_i$ for all $i=0,1,\ldots,d$,
    \item $A_iA_j = \sum_{h=0}^d p_{ij}^{h}A_h$ for all $i,j=0,1,\ldots,d$. 
\end{itemize}

Throughout this paper, all association schemes are symmetric. 
We note that our definition is in terms of adjacency matrices, whereas other --- equivalent --- definitions use the relations (or graphs) on $X$ that they represent. We will use the term relation also, but both for the relation itself as for its adjacency matrix. Occasionally and informally, we will call an association scheme just a {\em scheme}. For basic properties of association schemes, we refer to \cite{BBIT, bi, martintanaka}. We mention that $\mathcal{R}$ spans the so-called Bose-Mesner algebra, which is closed under both ordinary and entrywise multiplication. This algebra also has a basis of minimal {\em idempotents} $\{E_j : j=0,1,\ldots,d\}$, and there are so-called {\em Krein parameters} $q_{ij}^h$ such that 
\begin{itemize}
    \item $vE_0 = J$,
    \item $E_0+E_1+\cdots+E_d=I$, 
    \item $E_j^{\top}=E_j$ for all $j=0,1,\ldots,d$,
    \item $vE_i \circ E_j = \sum_{h=0}^d q_{ij}^{h}E_h$ for all $i,j=0,1,\ldots,d$. 
\end{itemize}
For completeness, we mention that the definitions impose that $A_i \circ A_j = \delta_{i,j}A_i$ and $E_iE_j = \delta_{i,j}E_i$ for all $i$ and $j$.
Furthermore there are the {\em first eigenmatrix} $P$ and the {\em second eigenmatrix} $Q$ containing eigenvalues and dual eigenvalues, respectively, that is, 
\begin{itemize}
    \item $A_i = \sum_{j=0}^d P_{ji}E_j$ for all $i=0,1,\ldots,d$,
    \item $E_j = \frac{1}{v}\sum_{i=0}^d Q_{ij}A_i$ for all $j=0,1,\ldots,d$.
\end{itemize}

In particular, we have that $P_{j0}=1$ for all $j$ and $P_{0i}=p_{ii}^0=:k_i$, the {\em valency} of relation $A_i$ for all $i$. Dually, we have that $Q_{i0}=1$ for all $i$ and $Q_{0j}=q_{jj}^0=:m_j$, the {\em rank} of idempotent $E_j$
for all $j$. 
We also have $PQ = vI$. This implies that both eigenmatrices $P$ and $Q$ have full rank, which we will use in the proof of Lemma~\ref{lem:rowscolsP}.
The {\em principal part} of $P$ and $Q$ is obtained by removing the first row and column; these parts contain the {\em restricted eigenvalues} and {\em restricted dual eigenvalues}.

We will make frequent use of the following lemma about the principal part of the eigenmatrices of an association scheme.

\begin{lemma}\label{lem:rowscolsP} Let $M$ be one of the eigenmatrices of an association scheme. Let $t \geq 2$. For any $t$ rows of the principal part of $M$, there are at least $t$ columns of the principal part of $M$ that are not constant on the $t$ rows.
\end{lemma}

\begin{proof} Suppose that there are $d-t+2$ columns that are constant on the $t$ rows. We first consider the submatrix $M'$ on the $t$ rows, and claim that it has rank at most $t-1$. Indeed, because $M'$ has constant row sums, the $t \times (t-1)$ submatrix $M''$  of $M'$ on these rows and the remaining $t-1$ columns has constant row sums, which implies that $\Col M' =\Col M''$, with dimension clearly at most $t-1$. Thus, $M$ cannot have full rank, which is a contradiction. 
\end{proof}

\subsection{Fusions}
 
For subsets of indices $\I$ and $\J$, we let $A_{\I}=\sum_{i \in \I} A_i$ and similarly $E_{\J}=\sum_{j \in \J} E_j$.
Let $\mathcal{R} = \{A_i : i= 0,1,\ldots,d\}$ be an association scheme. 
Let $\pi = \{\pi(i) : i=0,1,\ldots,d'\}$ be a partition of $\{ 0,1,\ldots,d \}$ with $\pi(0) = \{ 0 \}$. If $\mathcal{R}_{\pi} = \{A_{\pi(i)} : i=0,1,\ldots,d'\}$ is an association scheme, we say that $\mathcal{R}_{\pi}$ is a {\em fusion scheme} of $\mathcal{R}$. 

We say that a tuple of non-trivial relations in an association scheme {\em fuses} if fusing that tuple of relations results in an association scheme.

The {\em Bannai-Muzychuk criterion} \cite{Bsub} \cite[Lemma~2.48]{BBIT} is a very useful way to check whether a partition of relations gives rise to a fusion scheme: a given partition $\pi = \{\pi(i) : i=0,1,\ldots,d'\}$ of $\{0,1,\ldots,d\}$  with $\pi(0)=\{0\}$ gives rise to a fusion scheme $\{A_{\pi(i)}:i=0,\dots,d'\}$ if and only if there is a --- unique --- partition $\rho = \{ \rho(j) : j=0,1,\ldots,d'\}$ of $\{0,1,\ldots,d\}$ with $\rho(0) = \{0\}$ such that each $(\rho(j), \pi(i))$-block of the first eigenmatrix $P$ has constant row sums. If so, then the latter row sum equals the $(j,i)$-entry of the first eigenmatrix of the fusion scheme, which has idempotents $\{E_{\rho(j)}:j=0,\dots,d'\}$. 

As a fusion of relations gives rise to a fusion of idempotents, we can also speak of fusions in terms of idempotents. We thus say that a tuple of idempotents  {\em fuses} if fusing that tuple results in an association scheme. In the remainder of the paper, whenever we mention fusion of relations or idempotents, we always mean non-trivial ones.

We note that the Bannai-Muzychuk criterion also applies to the second eigenmatrix $Q$ \cite[Lemma 1]{Muzy}.

For the sake of readability, we next introduce some notation for the correspondence between fusions of relations and fusions of idempotents.

For a partition $\pi$ of the indices of relations that gives rise to a fusion scheme, we let $\rho$ be the corresponding partition of the indices of idempotents. We then adopt the notation 
\[{\I_1},\dots , {\I_{p}} \correspond {\J_1},\dots ,{\J_{q}}\]
to represent this correspondence, where $\I_1,\ldots,\I_p$ are all the parts of $\pi$ with size at least $2$, and $\J_1,\ldots,\J_q$ are all the parts of $\rho$ with size at least $2$. 

\subsection{Strongly regular graphs}
A graph with $v$ vertices is called {\em strongly regular} with parameters $(v,k,\lambda,\mu)$ if it is non-empty, non-complete, is regular with valency $k$, and any two distinct vertices have $\lambda$ or $\mu$ common neighbors, depending on whether the two vertices are adjacent or not. The complement $A_2$ of a strongly regular graph $A_1$ is also strongly regular, and $\{I,A_1,A_2\}$ forms a $2$-class association scheme. Consequently, a strongly regular graph has two restricted eigenvalues.

For each positive integer $n$, we denote the set $\{1,\ldots,n\}$ by $[n]$.
We now recall a lemma that applies to fusing two strongly regular relations in an association scheme.

\begin{lemma}{\cite[Lemma~2]{D3}}\label{lem:railway}
    Let $\Gamma_1$ and $\Gamma_2$ be edge-disjoint strongly regular graphs on $v$ vertices. For each $i\in[2]$, let $A_i$ be the adjacency matrix, $k_i$ be the valency, and $a_i, b_i$ be the restricted eigenvalues of $\Gamma_i$. If $A_1A_2 = A_2A_1$, then $A_1+A_2$ has restricted eigenvalues $\theta_1 = a_1 + a_2$, $\theta_2 = a_1+b_2$, $\theta_3 = b_1+a_2$ and $\theta_4 = b_1 + b_2$ with respective restricted multiplicities 
    \small{\begin{align*}
        m_1 &= \frac{vb_1b_2 - (k_1-b_1)(k_2-b_2)}{(a_1-b_1)(a_2-b_2)}, &
        m_2 &= -\frac{vb_1a_2 - (k_1-b_1)(k_2-a_2)}{(a_1-b_1)(a_2-b_2)},\\
        m_3 &= -\frac{va_1b_2 - (k_1-a_1)(k_2-b_2)}{(a_1-b_1)(a_2-b_2)}, &
        m_4 &= \frac{va_1a_2 - (k_1-a_1)(k_2-a_2)}{(a_1-b_1)(a_2-b_2)}.
    \end{align*}}
    If $a_i > b_i$ for $i\in[2]$, then $m_2 > 0$ and $m_3 > 0$. 
\end{lemma}

A strongly regular graph with $v$ vertices and valency $k$ is of \emph{Latin square type} if $v = n^2$, $k = t(n-1)$, and the restricted eigenvalues are $n-t$, and $-t$, for some positive integers $n$ and $t$. It is of {\em negative Latin square type} if $n$ and $t$ are negative integers. We say that a set of graphs/relations are (strongly regular) of the {\em same type} if they are all strongly regular of Latin square type or they are all strongly regular of negative Latin square type. A {\em conference graph} is a strongly regular graphs with $k=\frac{1}{2}(v-1)$ and restricted eigenvalues $-\frac12 \pm \frac12 \sqrt{v}$. A conference graph on a square number of vertices $v$ is both of Latin square type and of negative Latin square type. In fact, we note that such conference graphs are the only strongly regular graphs that are both of Latin square type and of negative Latin square type. Therefore, we say that a strongly regular graph is of \emph{strictly} (negative) Latin square type if it is of (negative) Latin square type but it is not a conference graph. 

The following follows from Lemma \ref{lem:railway}. 

\begin{lemma}\label{lem:sharingeigenvaluesLS} Let $A_1$ and $A_2$ be two edge-disjoint commuting strongly regular graphs. If $A_1$ is of strictly Latin square type and $A_2$ is of strictly negative Latin square type, then $A_1$ and $A_2$ share eigenvectors for all four combinations of restricted eigenvalues.

If $A_1$ and $A_2$ are both of Latin square type (respectively negative Latin square type), then they do not share an eigenvector for their positive (respectively negative) restricted eigenvalues.
\end{lemma}

Here sharing an eigenvector for a combination $(\theta_1,\theta_2)$ of restricted eigenvalues of $A_1$ and $A_2$ means that there is a vector that is eigenvector for $A_1$ with eigenvalue $\theta_1$ and eigenvector for $A_2$ with eigenvalue $\theta_2$.

This lemma will allow us to give a stronger characterization of amorphic schemes in Theorem \ref{thm:LStype}, in the sense that we allow a mixture of Latin square type and negative Latin square type graphs.

For completeness we also mention the following well-known facts.

\begin{lemma}\label{lem:union}
    \begin{itemize}
        \item The complement of a strongly regular graph of a (negative) Latin square type is a strongly regular graph of the same type. 
        \item The union of two edge-disjoint and commuting strongly regular graphs of the same type is again a strongly regular graph of the same type.
    \end{itemize}
\end{lemma}

\begin{proof}
The first claim is easily checked using the eigenvalues. 

As for the second claim, let $A_1$ and $A_2$ be two such graphs with parameters $v_i = n^2$, $k_i = t_i(n-1)$, and restricted eigenvalues $a_i = -t_i$ and $b_i = (n-t_i)$ for all $i\in \{1,2\}$. 
By Lemma \ref{lem:railway} or the second part of Lemma~\ref{lem:sharingeigenvaluesLS}, the restricted eigenvalues of $A_1+A_2$ are $a_1 + b_2 = a_2 + b_1 = n - (t_1+t_2)$ and $a_1+a_2 = -(t_1 + t_2)$. Note that the valency of $A_1+A_2$ is $k_1+k_2 = (t_1+t_2)(n-1)$. Therefore, $A_1+A_2$ is a strongly regular graph of the same type. 
\end{proof}

Finally, we need the following lemma.

\begin{lemma}\label{lem:SRGkandaNLS}
    Let $\Gamma$ be a strongly regular graph on $n^2$ vertices and valency $k$, having a restricted eigenvalue $a$, such that $k=-a(n-1)$. Then $\Gamma$ is of  Latin square type or negative Latin square type. 
\end{lemma}

\begin{proof}
Let $\Gamma$ have parameters $(v,k,\lambda,\mu)$ and let $b$ be its other restricted eigenvalue. Then $v=n^2$ and it follows from standard identities for strongly regular graphs that $k-1-\lambda=-1-a-b-ab$. On the other hand,
$k-1-\lambda=\mu(v-1-k)/k=-(k+ab)(n+1+a)/a=(n-1-b)(n+1+a)$. By combining the two expressions, we obtain that $n=b-a$, which shows that $\Gamma$ is of Latin square type or negative Latin square type.
\end{proof}

We recall that the sign of $n$ determines the type of $\Gamma$. 

\subsection{Amorphic schemes}

We call $\mathcal{R}$ \emph{amorphic} if $\mathcal{R}_{\pi}$ is an association scheme for all partitions $\pi$ of $\{ 0,1,\ldots,d \}$ with $\pi(0) = \{ 0 \}$. Clearly, all relations (and unions of relations) are strongly regular.
Amorphic association schemes are formally self-dual, i.e., $P=Q$, possibly after rearranging the idempotents. 

We say that a square matrix has a \emph{canonical form} if via permutations of the rows and permutations of the columns it can be made into a matrix of the form 
    \[{\small \begin{bmatrix}
         b_1 & a_2 & a_3 & \cdots & a_{d} \\
         a_1 & b_2 & a_3 & \cdots & a_{d} \\ 
         a_1 & a_2 & b_3 & \cdots & a_{d} \\ 
         \vdots & \vdots & \ddots & \vdots \\
         a_1 & a_2 & a_3 & \cdots & b_{d} \\
    \end{bmatrix}}. \]

According to \cite[Prop.~2]{vDM2010}, an association scheme is amorphic if and only if the principal part of one of its eigenmatrices has a canonical form. 

\section{All pairs}\label{sec:pairs}

In this section, we will show that if all pairs of relations fuse, then the association scheme is amorphic. This is our basic starting result, which will be strongly improved in Section \ref{sec:notapath}.

\begin{lemma}\label{lem:pairs} Let $\mathcal{R}$ be an association scheme. Then there is a one-one correspondence between fusing pairs of relations and fusing pairs of idempotents.
\end{lemma}

\begin{proof} By the \BM, there is a pair of idempotents that fuses for every pair of relations that fuses, and vice versa. Different fusing pairs of relations give rise to different fusion schemes, and hence to different fusing pairs of idempotents, which shows the statement.
\end{proof}

\begin{thm}\label{thm:pairs}
Let $\mathcal{R}$ be an association scheme. If all pairs of relations fuse, then $\mathcal{R}$ is amorphic.
\end{thm}

\begin{proof} Let $d$ be the number of non-trivial relations in the scheme, and assume that all pairs of relations fuse. By Lemma \ref{lem:pairs} also all pairs of idempotents fuse, and there is a one-one correspondence between pairs of relations and pairs of idempotents. If $d = 2$ or $d=3$, then the statement is trivially true. Thus, we may assume that $d \geq 4$. 

Let us now consider the fusion of $\{A_1,A_2\}$. Without loss of generality, we may assume that $\{1,2\} \correspond \{1,2\}$. By the \BM, the eigenvalues of $A_1+A_2$ on $E_1$ and $E_2$ should be the same, i.e., $P_{11}+P_{12}=P_{21}+P_{22}$. Moreover, $A_i$ should have the same eigenvalue on $E_1$ and $E_2$, i.e., $P_{1i}=P_{2i}$, for all $i \geq 3$. Because $P$ is non-singular (hence it has no two equal rows; cf.~Lemma \ref{lem:rowscolsP}), it follows that $P_{11} \neq P_{21}$ and $P_{12} \neq P_{22}$. Similarly, it follows that every pair of rows of the principal part of $P$  differs in precisely two positions.

Consider next a third idempotent, without loss of generality $E_3$. By Lemma \ref{lem:rowscolsP} (with $t=3$), there is at least one relation, without loss of generality $A_3$, for which $P_{33} \neq P_{13} =P_{23}$. Suppose now that $P_{31}\neq P_{11}$ and $P_{31}\neq P_{21}$.
Because every two rows differ in exactly two positions, it then follows that $P_{32}=P_{12}$ and $P_{32}=P_{22}$, which contradicts $P_{12}\neq P_{22}$. Thus, without loss of generality, we may assume that $P_{31}= P_{21}$, which moreover implies that $P_{31} \neq P_{11}$, so $P_{32}=P_{12}$. Besides this, we then have that $P_{1i}=P_{2i}=P_{3i}$ for $i \geq 4$. By the \BM, it now follows that the triple of idempotents $\{E_1,E_2,E_3\}$  fuses, and $\{ 1,2,3 \} \correspond \{ 1,2,3 \}$. 

More generally, it follows that the triple of relations $\{A_1,A_2,A_j\}$ fuses, and $\{ 1,2,j \} \correspond \{ 1,2,j \}$ (at least, without loss of generality). Moreover, we know the structure of the corresponding $3 \times 3$ submatrix, $P^{12j}$ say, of $P$: for any two of its rows the two positions in $P$ where these rows differ are in this submatrix. We will next exploit this and show that $A_1$ is a strongly regular graph, and that the eigenvalue $P_{11}$ occurs only on idempotent $E_1$ (and possibly $E_0$), i.e., that $P_{j1}=P_{21}$ for $j \geq 2$.

Indeed, take $j \geq 4$, and let us instead assume that $P_{j1}=P_{11}$. Then $P_{j2}\neq P_{12}$ and $P_{jj}\neq P_{1j}$ (by the property of $P^{12j}$). Because every pair of rows differs in precisely two positions, it follows that $P_{j3}=P_{13}$. But then $P_{j3} \neq P_{33}$. Moreover, $P_{j1}=P_{11} \neq P_{21}=P_{31}$. So rows $j$ and $3$ differ in positions $1$ and $3$. But then $P_{j2}=P_{32}=P_{12}$, but that is a contradiction, which shows our claim that $P_{j1}=P_{21}$.

Because $A_1$ was chosen arbitrarily, it follows that each of the non-trivial relations $A_i$ is a strongly regular graph and that the eigenvalue $P_{ii}$ occurs for $A_i$ only on idempotent $E_i$ (and possibly $E_0$), possibly after rearranging the idempotents. Because the pair $\{A_i,A_j\}$ fuses, it follows easily that $\{ i,j \} \correspond \{ i,j \}$ and that  $P_{ij}+P_{ii}=P_{jj}+P_{ji}$ for $i,j \neq 0$. By \cite[Prop.~2]{vDM2010}, it now follows that $\mathcal{R}$ is amorphic.
\end{proof}

\section{Schemes with many (negative) Latin square type relations}
\label{sec:latinsquaresrelations}

Ivanov \cite{GIK} showed that the relations in an amorphic $d$-class association scheme with $d \geq 3$ are all strongly regular of the same type. On the other hand, Ito, Munemasa, and Yamada \cite{IMY} showed that any scheme in which all relations are of Latin square type or all relations are of negative Latin square type, is amorphic. 
We will now strengthen this result.

\begin{thm}\label{thm:LStype}
    Let $\mathcal{R}$ be an association scheme. If there is at most one relation in $\mathcal{R}$ that is neither strongly regular of Latin square type nor strongly regular of negative Latin square type, then $\mathcal{R}$ is amorphic. 
\end{thm}

\begin{proof} We note that if all but possibly one relation is of the same type, then it follows from Lemma~\ref{lem:union} that the remaining relation, which is the complement of the union of all others must also be of the same type, and hence the scheme is amorphic \cite{IMY}. We may thus assume that there is at least one relation of strictly Latin square type, say $A_1$, and at least one relation of strictly negative Latin square type, say $A_2$.

Suppose now that one of the other relations is a conference graph. By Lemma \ref{lem:sharingeigenvaluesLS} (which we will use over and over again without mentioning), there is an idempotent $E$ for which $A_1$ has a positive eigenvalue and $A_2$ a negative eigenvalue. As the conference graph is of the same type as $A_1$, its eigenvalue on $E$ must be negative. But it is also of the same type as $A_2$, so its eigenvalue on $E$ must be positive. Thus we have a contradiction, and hence none of the relations is a conference graph. 

Suppose then that there are $d_1$ relations of strict Latin square type, $d_2$ relations of strict negative Latin square type, and $d_3$ remaining ones (hence $d_3 \le 1$).

Note that because a strict Latin square type relation and a strict negative Latin square type relation share eigenvectors for all four combinations of restricted eigenvalues, it is clear that $d \ge 4$.

First, suppose now that $d_1=1$, and consider the corresponding relation $A_1$ of strictly Latin square type. For each relation $A_i$ of strictly negative Latin square type, there is at least one idempotent which has a negative eigenvalue for $A_i$ and a positive eigenvalue for $A_1$ and also at least one with a negative eigenvalue for $A_i$ and a negative for $A_1$. As different $A_i$ do not share idempotents for negative eigenvalues, we thus obtain at least $2d_2$ idempotents.
Furthermore, there must be at least one idempotent that has a positive eigenvalue for all of the strictly negative Latin square type relations; this follows because fusing all remaining ($d_1+d_3$) relations gives an amorphic scheme, in which there is such an idempotent. Thus it follows that $d \ge 2d_2+1$, and hence $d_3 \ge d_2 \ge 2$, which is a contradiction.

Thus, we may assume that $d_1 \ge 2$, and similarly, that $d_2 \ge 2$. With the same arguments as above, it follows that there is at least one idempotent for each pair of relations, with one of Latin square type, and one of negative Latin square type, where the eigenvalue for the first is positive, and for the second is negative, and that these idempotents must be different for each pair. Thus, in this case, $d \ge d_1d_2+1$, which implies that $d_3 \geq (d_1-1)(d_2-1)$. We therefore have a contradiction, unless possibly $d_1=d_2=2$, and $d_3=1$.

In the latter case, considering the above arguments, it follows that the principal part of the eigenmatrix $P$ is of the form
$${\small \begin{bmatrix}
    b_1 & a_2 & b_3 & a_4 & a_5\\
    b_1 & a_2 & a_3 & b_4 & b_5\\
    a_1 & b_2 & b_3 & a_4 & c_5\\
    a_1 & b_2 & a_3 & b_4 & d_5\\
    a_1 & a_2 & a_3 & a_4 & e_5
\end{bmatrix}},$$
where the first two relations are of Latin square type, and the next two are of negative Latin square type. This implies that $a_1+b_2=b_1+a_2$ and $a_3+b_4=b_3+a_4$, and now it easily follows that $P$ is singular, which is our final contradiction.   
\end{proof}

\section{Fusion and contraction in the fusing-relations graph}\label{sec:structurefpg}

Let $\mathcal{R}$ be a $d$-class association scheme and let $\Gamma = \AGraph(\mathcal{R})$ be its fusing-relations graph, that is, the graph with vertex set $\{1,2,\ldots,d\}$ such that $i$ and $j$ are joined by an edge if the pair $\{A_i,A_j\}$ fuses.  Thus, if $\{i,j\}$ is an edge in $\Gamma$, then we can obtain an association scheme $\mathcal{R}'$ by fusing the relations $A_i$ and $A_j$. We denote by $\fuse_{ij}(\mathcal{R})$ the fusing-relations graph of $\mathcal{R}'$. 
For ease of notation, we let the vertex set of $\fuse_{ij}(\mathcal{R})$ be $[d]\setminus \{i,j\} \cup \{ij\}$ such that the vertex $ij$ corresponds to the relation $A_{\{i,j\}}$. 

Let $\Gamma$ be a graph and let $\{i,j\}$ be an edge in $\Gamma$. The \emph{$ij$-contraction} of $\Gamma$, denoted by $\Gamma/ij$, is the simple graph obtained from $\Gamma$ by contracting the edge $\{i,j\}$, i.e., we replace vertices $i$ and $j$ by one new vertex $ij$, and for each $h \neq i,j$, if at least one of $\{i,h\}$ and $\{j,h\}$ is an edge in $\Gamma$, we replace these by one edge $\{ij,h\}$. 

\begin{lemma}\label{lem:hemlock}
    Let $\mathcal{R}$ be a $d$-class association scheme with relations $A_1,\ldots,A_d$ and let $\Gamma = \AGraph(\mathcal{R})$ be its fusing-relations graph. Assume that $\{i,j\}$ is an edge in $\Gamma$. Then $\fuse_{ij}(\mathcal{R})$ contains $\Gamma/ij$ as a subgraph. 
\end{lemma}

\begin{proof}
    Let $P$ and $P'$ be the eigenmatrices of $\mathcal{R}$ and $\mathcal{R}'$, respectively. 
    Let $i'$ and $j'$ be such that $\{i,j\} \correspond \{i',j'\}$ in $\mathcal{R}$. 
    Let $\{p,q\}$ be an edge in $\Gamma/ij$. We aim to prove that $\{p,q\}$ is also an edge in $\fuse_{ij}(\mathcal{R})$. 

    If $q=ij$, then we may assume without loss of generality that $\{p,i\}$ is an edge in $\Gamma$. Let $p'$ and $q'$ be such that $\{p,i\} \correspond \{p',q'\}$ in $\mathcal{R}$. Thus, $P_{p'p}+P_{p'i}=P_{q'p}+P_{q'i}$ and $P_{p'h}=P_{q'h}$ for $h \neq p,i$. We note that $\{p',q'\}\neq \{i',j'\}$ by the one-to-one correspondence of pairs; see Lemma \ref{lem:pairs}. This means that $p'$ and $q'$ correspond to different idempotents in $\mathcal{R'}$. In case $p'$ or $q'$ equals $i'$ or $j'$, then we let it correspond to $i'j'$ in $\mathcal{R'}$ in below equations. For example, when $p'=i'$ and $h \neq ij$, then $P'_{p'h} = P'_{(i'j')h}=P_{i'h}=P_{p'h}$.
    Then it follows (for any $p'$) that $$P'_{p'p}+P'_{p'(ij)}=P_{p'p}+P_{p'i}+P_{p'j}=P_{q'p}+P_{q'i}+P_{q'j}=P'_{q'p}+P'_{q'(ij)}$$ and $P'_{p'h}=P_{p'h}=P_{q'h}=P'_{q'h}$ for $h \neq p,ij$, which shows that $\{p,ij\} \correspond \{p',q'\}$ in $\mathcal{R'}$, and hence that 
     $\{p,ij\}$ is an edge in $\fuse_{ij}(\mathcal{R})$.

     Similarly (and more easily) it can be shown that $\{p,q\}$ is an edge in $\fuse_{ij}(\mathcal{R})$ if $p,q \neq ij$.
\end{proof}

The following example shows that it can happen that $\Gamma/ij$ is a proper subgraph of $\fuse_{ij}(\mathcal{R})$. 

\begin{exam}
    Let $\mathcal{R}$ be a $3$-class association scheme with  relations $A_1, A_2$ and $A_3$, such that $A_1$ is strongly regular, while $A_2$ and $A_3$ are not (such schemes exist). Clearly, the two relations $A_2$ and $A_3$ fuse and the fusing-relations graph $\Gamma$ of $\mathcal{R}$ is $K_1 \sqcup K_2$. If we fuse the relations $A_2$ and $A_3$, we obtain a $2$-class association scheme whose fusing-relations graph $\fuse_{2,3}(\mathcal{R})$ is $K_2$, because both relations are strongly regular. However, the contraction $\Gamma/23$ is $K_1 \sqcup K_1$. 
\end{exam}

\section{Not a path}\label{sec:notapath}

We are now ready to move to our main result, in which we state that if the fusing-relations graph is connected but not a path, then the corresponding association scheme is amorphic. Recall that in Example \ref{ex:path}, we constructed examples of association schemes for which the fusing-relations graph is a path (and hence these are not amorphic). 

Our first step towards this result is to consider cycles, or more generally, Hamiltonian graphs.

\begin{prop}\label{thm:Hcycle}
    Let $\mathcal{R}$ be a $d$-class association scheme with $d\geq 3$ and let $\Gamma = \AGraph(\mathcal{R})$ be its fusing-relations graph. If $\Gamma$ is Hamiltonian, then $\mathcal{R}$ is amorphic.     
\end{prop}

\begin{proof}
    We apply induction on $d$. For $d=3$, the result is clear (Theorem~\ref{thm:pairs}).     
    Now let $d\geq 4$ and assume that the result is valid for $(d-1)$-class schemes. 
    Consider a Hamiltonian cycle of $\Gamma$ and two consecutive vertices in this cycle. By Lemma~\ref{lem:hemlock}, we can fuse these two vertices and obtain a 
    $(d-1)$-class scheme that is Hamiltonian, and hence it is amorphic. Thus, its relations are either all of Latin square type or all of negative Latin square type. Among these relations are $d-2$ relations of $\mathcal{R}$. Since we can do this for any two consecutive vertices of the $d$-cycle, it follows that all relations are of the same type, and hence $\mathcal{R}$ is amorphic.      
\end{proof}

Next, we consider $K_{1,3}$.

\begin{lemma}\label{lem:K13}
    Let $\mathcal{R}$ be a $4$-class association scheme with fusing-relations graph $\Gamma = \AGraph(\mathcal{R})$ containing $K_{1,3}$. Then $\mathcal{R}$ is amorphic.
\end{lemma}
    
\begin{proof}
    Assume that the edge set of $\Gamma$ contains $12,13,$ and $14$.
    Without loss of generality we may assume that $\{1,2\} \correspond \{1,2\}$
    and $\{1,3\} \correspond \{1,3\}$, and hence the eigenmatrix $P$ has the form 
\[P={\small \begin{bmatrix}
    1 & k_1 & k_2 & k_3 & k_4 \\
    1 & b_1 & a_2 & a_3 & a_4 \\ 
    1 & a_1 & b_2 & a_3 & a_4 \\ 
    1 & \ast & a_2 & b_3 & a_4 \\
    1 & \ast & \ast & \ast & b_4 \\
\end{bmatrix}}. \]

It is clear that $\{1,4\} \correspond \{2,3\}$ is not possible. 
Suppose that $\{1,4\} \correspond \{3,4\}$
 (or similarly $\{1,4\} \correspond \{2,4\}$). Then 
\[P={\small \begin{bmatrix}
    1 & k_1 & k_2 & k_3 & k_4 \\
    1 & b_1 & a_2 & a_3 & a_4 \\ 
    1 & a_1 & b_2 & a_3 & a_4 \\ 
    1 & \ast & a_2 & b_3 & a_4 \\
    1 & \ast & a_2 & b_3 & b_4 \\
\end{bmatrix}}. \]
For two real numbers $a$ and $b$, by $a \diamond b$ we denote that one of $a$ and $b$ is nonnegative and the other is negative. 
This occurs precisely when $a$ and $b$ are the restricted eigenvalues of a strongly regular graph, hence it follows that $a_i \diamond b_i$ for $i = 2,3,4$.
As mentioned before, by Lemma~\ref{lem:railway}, an association scheme having two strongly regular relations has idempotents on which one of the strongly regular relations has a nonnegative restricted eigenvalue, whereas the other strongly regular relation has a negative restricted eigenvalue.
Therefore, $b_2\diamond a_3$, $b_2\diamond a_4$, and $b_4\diamond b_3$, which leads to a contradiction. 

Thus, $\{1,4\} \correspond \{1,4\}$ and 
\[P={\small \begin{bmatrix}
    1 & k_1 & k_2 & k_3 & k_4 \\
    1 & b_1 & a_2 & a_3 & a_4 \\ 
    1 & a_1 & b_2 & a_3 & a_4 \\ 
    1 & \ast & a_2 & b_3 & a_4 \\
    1 & \ast & a_2 & a_3 & b_4 \\
\end{bmatrix}}. \]
For each pair $i,j \in \{2,3,4\}$, the strongly regular graphs $A_i$ and $A_j$ do not share eigenvectors for $b_i$ and $b_j$ and hence the restricted multiplicity of the restricted eigenvalue $b_i + b_j$ of $A_i+A_j$ must be zero. By Lemma~\ref{lem:railway}, it now follows that $va_2a_3=(k_2-a_2)(k_3-a_3)$,
$va_2a_4=(k_2-a_2)(k_4-a_4)$, and $va_3a_4=(k_3-a_3)(k_4-a_4)$.

Note first that this system of three equations does not allow for any of the $a_i$ to be zero. Thus, it follows that 
$$\frac{k_2-a_2}{a_2}=\frac{k_3-a_3}{a_3} =\frac{k_4-a_4}{a_4}= \pm \sqrt{v}. $$

Observe now that $\sqrt{v}$ is a rational number, and so $v$ is a square. By Lemma \ref{lem:SRGkandaNLS}, it now follows that $A_2,A_3$, and $A_4$ are strongly regular of Latin square type or negative Latin square type, and hence by Theorem \ref{thm:LStype}, $\mathcal{R}$ is amorphic.
\end{proof}

By combining the above, we can now prove our main result.

\begin{thm}\label{thm:mainresult}
    Let $\mathcal{R}$ be a $d$-class association scheme with $d\geq 3$ and let $\Gamma = \AGraph(\mathcal{R})$ be its fusing-relations graph. If $\Gamma$ is a connected graph that is not a path, then $\mathcal{R}$ is amorphic. 
\end{thm}

\begin{proof}
    The case of $d=3$ is trivial as a result of Theorem~\ref{thm:pairs}. We assume $d\geq 4$ hereafter. 

    A connected graph that is not a path is either a cycle or it has a vertex of valency at least $3$. Because of Proposition~\ref{thm:Hcycle}, we only need to prove that $\mathcal{R}$ is amorphic if $\Gamma$ contains $K_{1,3}$. 
    To this end, we apply induction on $d$ and note that
    the case $d=4$ is Lemma~\ref{lem:K13}. 
    
    Let $d\geq 5$ and let $\Gamma$ contain $K_{1,3}$. We claim that we can always find two different edges $uv$ and $uv'$ such that both graphs $\Gamma/uv$ and $\Gamma/uv'$ are connected and contain $K_{1,3}$. Then by induction and Lemma~\ref{lem:hemlock}, we obtain that all relations but $A_u$ and $A_v$ are strongly regular of the same type, and that all relations but $A_{u}$ and $A_{v'}$ are also strongly regular of the same type. Now it follows from Theorem~\ref{thm:LStype} that $\mathcal{R}$ is amorphic. 

    It remains to prove the claim. If $\Gamma$ is a star, then any two edges show the claim. If $\Gamma$ is not a star, then besides the vertex of degree at least $3$, there is a vertex $u$ of degree at least $2$. Any two edges through $u$ then show the claim.
\end{proof}

In light of Example~\ref{ex:general-clique+vertex}, where we constructed examples of association scheme with fusing-relations graphs $K_{d-1} \sqcup K_1$, the minimal number of edges that the fusing-relations graph should have in order to be sure that the association scheme is amorphic is at least $\binom{d-1}{2}+1$. Clearly, this number is sufficient. 

\begin{cor}\label{cor:numberofedges}
Let $d\geq 3$ and $\mathcal{R}$ be a $d$-class association scheme. If $\mathcal{R}$ has more than $\binom{d-1}{2}$ fusing pairs of relations, then $\mathcal{R}$ is amorphic. 
\end{cor}

\section{Dualization}\label{sec:dualization}

So far, the results concern fusions of pairs of relations. In this section, we shall show that the same results hold for fusions of pairs of idempotents. It is clear from Lemma \ref{lem:pairs} and Theorem \ref{thm:pairs} that if all pairs of idempotents fuse, then the scheme is indeed amorphic. 

\begin{prop}\label{thm:dualpairs}
Let $\mathcal{R}$ be an association scheme. If all pairs of idempotents fuse, then $\mathcal{R}$ is amorphic.
\end{prop}

In order to dualize our main results, we however need to dualize also the tools that we used. 

\subsection{Strongly regular idempotents}

Let $\mathcal{R}$ be a $d$-class association scheme with relations $A_0,A_1,\ldots,A_d$, idempotents $E_0$, $E_1$, $\ldots$, $E_d$ and second eigenmatrix $Q$. Recall that for each $j\in[d]$, the entries $Q_{i,j}, i\in [d]$ are the restricted dual eigenvalues of $E_j$. Then an idempotent $E$ in $\mathcal{R}$ is called \emph{strongly regular} if $E$ has exactly two restricted dual eigenvalues.
We note that it follows from Lemma \ref{lem:rowscolsP} (with $t=d$) that an idempotent $E \neq E_0$ has at least two restricted dual eigenvalues, unless $d=1$.

We note also that by the dual Bannai-Muzychuk criterion, a strongly regular idempotent $E_j$ with restricted dual eigenvalues $a$ and $b$ generates a $2$-class fusion scheme, with non-trivial idempotents $E_j$ and $I-E_j-\frac1vJ$, and non-trivial, strongly regular relations $B_1 = \sum_{i>0: \ Q_{i,j} = a} A_i$ and $B_2 = \sum_{i>0: \ Q_{i,j} = b} A_i$. This implies the following for $a$ and $b$.

\begin{lemma}\label{lem:Smith}
    Let $E$ be a strongly regular idempotent with rank $m$ and dual restricted eigenvalues $a, b$ with $a > b$, and let $\mathcal{R}$ be the $2$-class association scheme with $E_1=E$. Then 
    \begin{itemize}
      \item $m \geq a \geq 0 > -1 \geq b \geq -m$, 
      \item
        $ab = q^2_{11} - m$,
        \item $a+b = q^1_{11} - q^2_{11}$.
    \end{itemize}
\end{lemma}

\begin{proof}
    Consider the strongly regular relation $B_1$ in $\mathcal{R}$, let $k$ be its valency, and let  $r$ and $s$, with $r > s$ be the restricted eigenvalues of $B_1$. Then $k\geq r \geq 0 > -1 \geq s \geq -k$. By \cite[p.~23]{BvMSRG}, we obtain that $\{a ,b\} = \{ \frac{m r}{k}, -\frac{m(r+1)}{v-k-1} \}$, and so $m \geq a\geq 0>-1\geq b \geq -m$.

    Secondly, $v E_1 \circ E_1 = q^0_{11} E_0 + q^1_{11} E_1 + q^2_{11} E_2 = q^2_{11} I + \frac{1}{v}(q^0_{11} - q^2_{11}) J + (q^1_{11} - q^2_{11}) E_1$.
    By multiplying this equation entrywise, once by $B_1$ and once by $B_2$, we obtain that
    $x^2 - (q^1_{11} - q^2_{11}) x - (q^0_{11} - q^2_{11}) = 0$ has solutions for $x = a,b$. Thus, $ab = q^2_{11} - q^0_{11} = q^2_{11} - m$ and $a+b = q^1_{11} - q^2_{11}$.
\end{proof}

We note that also Suda \cite[Theorem 4.1 (1)]{Suda} studied the dual eigenvalues, of $Q$-polynomial schemes. 
The dual of Lemma \ref{lem:railway}, in the context of association schemes, is as follows.

\begin{lemma}\label{thm:daze}
    Let $d \geq 2$ and let $\mathcal{R}$ be a $d$-class association scheme on $v$ vertices. Assume that $E_1$ and $E_2$ are two strongly regular idempotents of ranks $m_1$ and $m_2$, respectively, such that $\{ Q_{i,1}: \ i\in[d] \} = \{ a_1, b_1\}$ and $\{ Q_{i,2}: \ i\in[d] \} = \{ a_2, b_2\}$, with $a_j \neq b_j$ for $j=1,2$. Let 
    \begin{align*}
        B_1 &= \sum_{i>0: \ Q_{i,1} = a_1, Q_{i,2} = a_2} A_i, & B_2 &= \sum_{i>0: \ Q_{i,1} = a_1, Q_{i,2} = b_2} A_i, \\
            B_3 &= \sum_{i>0: \ Q_{i,1} = b_1, Q_{i,2} = a_2} A_i, & B_4 &= \sum_{i>0: \ Q_{i,1} = b_1, Q_{i,2} = b_2} A_i.
    \end{align*}
    For each $i\in[4]$, let $\ell_i$ be the integer such that $B_iJ = \ell_i J$. Then 
    \small{\begin{align}\label{eq:indent}
        \ell_1 &= \frac{vb_1b_2 - (m_1-b_1)(m_2-b_2)}{(a_1-b_1)(a_2-b_2)},&
        \ell_2 &= -\frac{vb_1a_2 - (m_1-b_1)(m_2-a_2)}{(a_1-b_1)(a_2-b_2)},\\
        \ell_3 &= -\frac{va_1b_2 - (m_1-a_1)(m_2-b_2)}{(a_1-b_1)(a_2-b_2)},&
        \ell_4 &= \frac{va_1a_2 - (m_1-a_1)(m_2-a_2)}{(a_1-b_1)(a_2-b_2)}. \notag
    \end{align}}
    If $a_i > b_i$ for $i\in[2]$, then $\ell_2 > 0$ and $\ell_3 > 0$. 
\end{lemma}

\begin{proof}
    First, we look at the last claim. 
    Applying Lemma~\ref{lem:Smith} for $j=1,2$, we obtain that $m_i \geq a_i\geq 0>-1\geq b_i \geq -m_i$ for $i\in[2]$. This proves that $\ell_2$ and $\ell_3$ are both positive. 
 
Now we turn to prove \eqref{eq:indent}. 
Since $B_1+B_2+B_3+B_4 = J - I$, it is clear that 
\begin{align}\label{eq:cholera}
    \ell_1+\ell_2+\ell_3+\ell_4 = v-1. 
\end{align}

Let $P$ be the first eigenmatrix of $\mathcal{R}$. It follows from $PQ = vI$ that 
\begin{align}\label{eq:fervent}
    m_1 + (\ell_1+\ell_2) a_1 + (\ell_3+\ell_4) b_1 = m_2 + (\ell_1+\ell_3) a_2 + (\ell_2+\ell_4) b_2 = 0. 
\end{align}
Because $q^0_{1,2}=0$, we have that
\begin{align}\label{eq:heedless}
     0&=(v E_1\circ v E_2) E_0  \\
      &=  (m_1m_2 I + a_1a_2 B_1 + a_1b_2 B_2 + b_1a_2 B_3 + b_1b_2 B_4) E_0 \notag \\ 
    &=(m_1m_2+ a_1a_2\ell_1 + a_1b_2\ell_2 + b_1a_2\ell_3 + b_1b_2\ell_4) E_0.     \notag 
\end{align}
Finally, \eqref{eq:indent} follows easily from the system of linear equations given by \eqref{eq:cholera},\eqref{eq:fervent}, and \eqref{eq:heedless}. 
\end{proof}

\subsection{Latin square type idempotents}\label{sec:LSidempotents}

Let $\mathcal{R}$ be a scheme on $v$ vertices.
We say that a strongly regular idempotent of rank $m$ in $\mathcal{R}$ is of Latin square type if $v=n^2$, $m=t(n-1)$, and the restricted dual eigenvalues are $n-t$ and $-t$, for some positive integers $n$ and $t$. It is of negative Latin square type if $n$ and $t$ are negative integers. As in the dual case, we call such an idempotent of strictly (negative) Latin square type if $m \neq \frac12(n^2-1)$, thus excluding being of both types. 

From these definitions and Lemma \ref{thm:daze}, it is clear now that also the dual of Lemma \ref{lem:sharingeigenvaluesLS} holds.
Finally, we need to show the dual of Lemma \ref{lem:SRGkandaNLS}.

\begin{lemma}\label{lem:dualSRGkandaNLS}
    Let $E$ be a strongly regular idempotent on $n^2$ vertices and rank $m$, having a dual restricted eigenvalue $a$, such that $m=-a(n-1)$. Then $E$ is of Latin square type or of negative Latin square type.
\end{lemma}

\begin{proof}
Let $v=n^2$. Consider the $2$-class association scheme generated by $E$ and let $b$ be its other restricted dual eigenvalue. 
Write $m_1 = m$ and $m_2 = v-1-m$. 
First, we note that by \cite[Proposition~2.24]{BBIT}, we have that
$q^2_{11}(v-1-m) = m_2 q^2_{11} = m_1 q^1_{21} = m(m-1-q^1_{11})$.
Now it follows from Lemma~\ref{lem:Smith} that $m-1-q^1_{11}=-1-a-b-ab=q^2_{11}(v-1-m)/m$. 
Similar as in the proof of Lemma \ref{lem:SRGkandaNLS} we now obtain that $E$ is of (negative) Latin square type, by using $m=-a(n-1)$.
\end{proof}

\subsection{Main dual results}

Finally, we note that Lemma \ref{lem:hemlock} can easily be dualized by replacing relations by idempotents and $P$ by $Q$.
Thus, all our tools are dualized, and we can conclude with the following main dual results.

\begin{prop}\label{thm:dualLStype}
    Let $\mathcal{R}$ be an association scheme. If there is at most one idempotent in $\mathcal{R}$ that is neither strongly regular of Latin square type nor strongly regular of negative Latin square type, then $\mathcal{R}$ is amorphic. 
\end{prop}

\begin{prop}\label{thm:dualmainresult}
    Let $\mathcal{R}$ be a $d$-class association scheme with $d\geq 3$ and let $\Gamma = \EGraph(\mathcal{R})$ be its fusing-idempotents graph. If $\Gamma$ is a connected graph that is not a path, then $\mathcal{R}$ is amorphic. 
\end{prop}

\section{Final remarks and problems}\label{sec:finalremarks}

We finish this paper with a few remarks and open problems.

First of all, we note that if one of the fusion-relations graph and fusing-idempotents graph of a given scheme is connected, then they must be isomorphic: either they are both paths or complete. It would be interesting to determine whether there are schemes for which the fusing-relations graph and fusing-idempotents graph are {\em not} isomorphic. In fact, what can be said  in general about the connected components of these graphs: are they all paths or cliques ?

Secondly, if two of the relations/idempotents of an association scheme are strongly regular and both of Latin square type or both of negative Latin square type, can they always be fused? A possible starting point would be to determine 4-class schemes on $n^2$ vertices that have two relations that represent two ``orthogonal"  partitions into $n$ cliques.

Finally, we note that our results seem to rely just on the algebraic properties of the
eigenmatrices and intersection numbers, so the results may apply to table algebras \cite{PB} as well.

\section*{Declaration of competing interest}

The authors declare that they have no conflict of interest.

\section*{Data availability}
No data was used for the research described in the article.

\section*{Acknowledgement} The authors thank the anonymous referees and Bill Martin for their detailed comments and suggestions. This work is supported by the National Key R. and D. Program of China (No. 2020YFA0713100), the National Natural Science Foundation of China (No. 12071454, 12471335) and the Anhui Initiative in Quantum Information Technologies (No. AHY150000). The work obtained was mostly done while Yanzhen Xiong was working as a postdoctoral fellow in the University of Science and Technology of China.

\end{document}